\documentclass[12pt]{article}
\usepackage[centertags]{amsmath}
\usepackage{amsfonts}
\usepackage{amssymb,latexsym}
\usepackage{amsthm}
\usepackage{newlfont}
\usepackage{graphicx}
\newfont{\bb}{msbm10 at 12pt}
\def\r{\hbox{\bb R}}
\def\h{\hbox{\bb H}}
\def\s{\hbox{\bb S}}
\def\e{\hbox{\bf{E}}}

\def\t{\hbox{\bf T}}
\def\n{\hbox{\bf N}}
\def\b{\hbox{\bf B}}

\newtheorem{theorem}{Theorem}[section]
\newtheorem{definition}[theorem]{Definition}
\newtheorem{proposition}[theorem]{Proposition}
\newtheorem{remark}[theorem]{Remark}
\newtheorem{corollary}[theorem]{Corollary}

\title{Constant angle surfaces in Minkowski space}
\author{ Rafael L\'opez\thanks{Partially
supported by MEC-FEDER
 grant no. MTM2007-61775 and
Junta de Andaluc\'{\i}a grant no. P06-FQM-01642.}\\
Departamento de Geometr\'{\i}a y Topolog\'{\i}a\\
Universidad de Granada, Spain\\
email: {\tt rcamino@ugr.es}\\
\\
Marian Ioan Munteanu\thanks{The author was supported by grant PN-II ID 398/2007-2010 (Romania)}\\
 Universitatea 'Al.I.Cuza' Ia\c si,
Facultatea de Matematic\u a\\
Bd. Carol I, n.11,
700506 Ia\c si,
Romania\\
email: {\tt marian.ioan.munteanu@gmail.com}}

\date{}
\begin{document}

\maketitle
\begin{abstract}
A constant angle surface in Minkowski space is a spacelike surface whose unit normal vector field makes a constant hyperbolic
angle with a fixed timelike vector. In this work we study and classify these surfaces. In particular, we show that they are flat.
Next we prove that a tangent developable surface (resp. cylinder, cone) is a constant angle surface if and only if the generating
curve is a helix (resp. a straight-line, a circle).\\[2mm]
{\bf Keywords and Phrases:} constant angle surfaces, Minkowski space, helix\\[2mm]
{\bf 2000 Mathematics Subject Classification:} 53B25
\end{abstract}

\section{Introduction and statement of results}

A constant angle surface in Euclidean three-dimensional space $\e^3$ is a surface whose tangent planes make a constant angle
with a fixed vector field of the ambient space. These surfaces generalize the concept of helix, that is, curves whose tangent
lines make a constant angle with a fixed vector of $\e^3$. This kind of surfaces are models to describe some phenomena in physics
of interfaces in liquids crystals and of layered fluids \cite{kn:CS}. Constant angle surfaces have been studied for arbitrary dimension
in Euclidean space $\e^n$ \cite{kn:DR,kn:R} and in different ambient spaces, e.g. $\s^2\times\r$, $\h^2\times\r$
and $\mbox{Nil}_3$ \cite{kn:DFVV07,kn:DM09,kn:FMV}.

In this work we extend the concept of constant angle surfaces in Lorentzian ambient space.
Let $\e_1^3$ denote the three-dimensional Minkowski space, that is, the real vector space $\r^{3}$ endowed with
the Lorentzian metric
$$\langle~,~\rangle=(dx_1)^2+(dx_2)^2-(dx_3^2),$$
where $(x_1,x_2,x_3)$ are the canonical coordinates in $\r^{3}$. In Minkowski space $\e_1^3$ and due to the variety
of causal character of a vector, there is not a natural concept of angle between two arbitrary vectors and only it is
possible to define the angle between {\em timelike} vectors.

Consider a (connected) surface $M$ and  a smooth immersion  $x: M\rightarrow \e_1^3$. We say that $x$ is a \emph{spacelike}
immersion if the induced metric on $M$  via $x$ is a Riemannian metric on each tangent plane. This is equivalent to say that
any unit normal vector field $\xi$ of $M$ is timelike at each point. In particular, if $x:M\rightarrow\e_1^3$ is a spacelike
immersion, then the surface $M$ is orientable.

\begin{definition}
Let {\sc $x:M\rightarrow\e_1^3$} be a spacelike immersion and let $\xi$ be a  unit normal vector field on $M$. We say that $M$ is a
constant angle surface if there is a fixed timelike  vector $U$ such that  $\xi$ makes a constant hyperbolic angle with $U$.
\end{definition}

In Theorem \ref{t-main} we do a local description of any constant angle spacelike surface. As a consequence,
we prove that they are ruled and flat surfaces (Corollary \ref{flat}). Thus they must be tangent developable
surfaces, cylinders and cones. In Section \ref{s-tangent} we deal with tangent surfaces showing in Theorem \ref{tangent} that

\begin{quote}{\em  A tangent developable surface is a constant angle surface if and only if the generating curve is a helix.}
\end{quote}
Finally we consider in Section \ref{s-cylinder} cylinders and cones. We show (see Theorems
\ref{t-cylinder} and \ref{t-cone})
\begin{quote}{\em  The only spacelike cylinders that are  constant angle surfaces are planes. A cone is a constant angle surface
if and only if the generating curve is a circle contained in a spacelike plane.}
\end{quote}

\section{Preliminaries} \label{preli}
Most of the following definitions can be found in the O'Neill's book \cite{kn:ON}.
Let $\e_1^3$ be the three-dimensional Minkowski space. A vector $v\in\e_1^3$ is said
spacelike if $\langle v,v\rangle>0$ or $v=0$,
timelike if $\langle v,v\rangle<0$,
and lightlike if $\langle v,v\rangle =0$ and $v\neq0$.
The norm (length) of a vector $v$ is given by $|v|=\sqrt{|\langle v,v\rangle|}$.

In Minkowski space $\e_1^3$ one can define the angle between two vectors only if both are timelike.  We describe this fact.
If  $u,v\in\e_1^3$ are two timelike vectors, then $\langle u,v\rangle\not=0$.
We say that $u$ and $v$ lie in the same timelike cone if $\langle u,v\rangle <0$.
This defines an equivalence binary relation with exactly two equivalence classes.
If $v$ lies in the same timelike cone than $E_3:=(0,0,1)$, we say that $v$ is future-directed.
For timelike vectors, we have the Cauchy-Schwarz inequality given by
$$ |\langle u,v\rangle|\geq \sqrt{-\langle u,u\rangle}\sqrt{-\langle v,v\rangle}$$
and the equality holds if and only if $u$ and $v$ are two proportional vectors.
In the case that both vectors lie in the same timelike cone, there exists a unique number $\theta\geq 0$ such that
$$\langle u,v\rangle=-|u| |v|\cosh(\theta).$$
This number $\theta$ is called  the \emph{hyperbolic angle} between $u$ and $v$.

\begin{remark} We point out that the above reasoning cannot work for other pairs of vectors, even if they are spacelike. For example,
the vectors $u=(\cosh(t),0,\sinh(t))$ and $v=(0,\cosh(t),\sinh(t))$ are spacelike vectors with $|u|=|v|=1$ for any $t$. However the number
$\langle u,v\rangle=-\sinh(t)^2$ takes arbitrary values from $0$ to $-\infty$. Thus, there is not $\theta\in\r$ such that
$\cos(\theta)=\langle u,v\rangle$.

\end{remark}

We also need to recall the notion of the Lorentzian cross-product $\times:\e^3_1\times\e^3_1 \rightarrow \e^3_1$.
If $u,v\in\e_1^3$, the vector $u\times v$ is defined as the unique one that satisfies $\langle u\times v,w\rangle=\mbox{det}(u,v,w)$,
where $\mbox{det}(u,v,w)$ is the determinant of the matrix whose columns are the vectors $u$, $v$ and $w$ with respect to the usual coordinates.
An easy computation gives
$$u\times v=
   (u_2v_3-u_3v_2, u_3v_1-u_1v_3,u_2v_1-u_1v_2).$$
As the cross-product in Euclidean 3-space, the Lorentzian cross-product has
 similar algebraic  properties, such as the anti-symmetric property or the fact that $u\times v$ is orthogonal both $u$ as $v$.

Let $x:M\rightarrow\e_1^3$ be an immersion of a surface  $M$ into $\e_1^3$. We say that $x$ is spacelike (resp. timelike, lightlike)
if the induced metric on $M$ via $x$ is Riemannian (resp. Lorentzian, degenerated). This is equivalent to assert that a (local) normal vector $\xi$ is
timelike (resp. spacelike, lightlike). As the concept of angle is given only for timelike vectors, we have to consider those immersions
whose unit normal vector is timelike, that is, {\it spacelike} immersions.
Let $x$ be a spacelike immersion. At any point $p\in M$, it is possible to choose a unit normal vector $\xi(p)$ such that $\xi(p)$
 is future-directed, i.e. $\langle\xi(p),E_3\rangle<0$. This shows that if $x$ is a spacelike immersion, the surface $M$ is orientable.

Denote $\mathfrak{X}(M)$ the space of tangent vector fields on $M$. Let $X, Y\in \mathfrak{X}(M)$. We write by  $\overset{\sim}{\nabla}$ and $\nabla$
the Levi-Civita connections of  $\e_1^3$ and $M$ respectively.
Moreover,
$$\nabla_X Y=(\overset{\sim}{\nabla}_X Y)^\top$$
where the superscript $^\top$  denotes the tangent part of the vector field $\nabla_X Y$. We define the second fundamental form of $x$ as
the tensorial, symmetric map $\sigma:\mathfrak{X}(M)\times\mathfrak{X}(M)\rightarrow {\mathfrak X}(M)^\bot$ given by
$$\sigma(X,Y)=(\overset{\sim}{\nabla}_X Y)^\bot$$
where by ${}^\bot$ we mean the normal part. The Gauss formula can be written as
\begin{equation}\label{gau}
\overset{\sim}{\nabla}_X Y=\nabla_X Y+\sigma(X,Y).
\end{equation}
We denote by $A_\xi (X)=A(X)$ the tangent component of
 $-\overset{\sim}{\nabla}_X \xi$, that is, $A_\xi(X)=-(\overset{\sim}{\nabla}_X \xi)^\top$.
Because  $\langle\overset{\sim}{\nabla}_X \xi,\xi\rangle=0$, we have  the so-called Weingarten formula
\begin{equation}\label{wei}
\overset{\sim}{\nabla}_X \xi=-A_\xi(X).
\end{equation}
 The map  $A:\mathfrak{X}(M)\rightarrow\mathfrak{X}(M)$ is called the \emph{Weingarten endomorphism} of the immersion
 $x$ and $\sigma$ is the second fundamental form of $x$. We have then
 $\langle AX,Y\rangle=\langle X,AY\rangle$. As a consequence, the Weingarten endomorphism is diagonalizable, that is, if $p\in M$, the map  $A_p:T_p M\rightarrow T_pM$ defined by
$A_p(v)=(AX)_p$ is diagonalizable, where  $X\in\mathfrak{X}(M)$ is a vector field that extends  $v$. The eigenvalues of $A_p$ are called the
{\it principal curvatures} and they  will be denoted by  $\lambda_i(p)$.
Moreover, if $X,Y\in \mathfrak{X}(M)$, we have $\langle A(X),Y\rangle=\langle\sigma(X,Y),\xi\rangle$ and
$$\sigma(X,Y)=-\langle\sigma(X,Y),\xi\rangle \xi=-\langle A(X),Y\rangle \xi.$$
$$\overset{\sim}{\nabla}_XY=\nabla_X Y-\langle A(X),Y\rangle \xi.$$

Let $\{v_1,v_2\}$ be a basis in the tangent plane $T_p M$ and we denote
$$\sigma_{ij}=\langle \sigma(v_i,v_j),\xi\rangle=\langle A(v_i),v_j\rangle.$$
If we assume that this basis is orthonormal, we have from \eqref{gau} and \eqref{wei}
\begin{equation}\label{vij1}
\overset{\sim}{\nabla}_{v_i}V_j=\nabla_{v_i}V_j-\sigma_{ij}\xi.
\end{equation}
\begin{equation}\label{vij2}
\overset{\sim}{\nabla}_{v_i}\xi=\sigma_{i1}v_1+\sigma_{i2}v_2.
\end{equation}
where $V_i$ is a tangent vector field that extends $v_i$.

\section{Classification of constant angle surfaces in $\e_1^3$}
\label{section-main}

Let $M$ be a constant angle spacelike surface in $\e_1^3$ whose unit normal vector field $\xi$ is assumed to be future-directed.
Without loss of generality, we assume that $U$ is a unitary vector and after an isometry of the ambient space, we can take $U$ as $E_3$.
Denote by $\theta$ the hyperbolic angle between $\xi$ and $U$,
 that is, $\cosh(\theta)=-\langle\xi,U\rangle$. If $\theta=0$, then
 $\xi=U$ on $M$. This means that  $x$ describes the  immersion of an affine plane parallel to $Ox^1x^2$.
 Throughout this work, we discard the trivial case that $\theta=0$.

We decompose $U$ as
$$U=U^\top+\cosh(\theta) \xi
$$
where $U^\top$ is the projection of $U$ on the tangent plane of $M$. Let
$$e_1=\frac{U^\top}{|U^\top|},$$
which defines a unit tangent vector field on $M$ and consider $e_2$ a unit  vector field on $M$ orthogonal to $e_1$ in such a way that
$\{e_1,e_2,\xi\}$ defines an oriented unit orthonormal  basis for every point of $M$. We write now the vector $U$ in the following form
\begin{equation}
\label{uu}
U=\sinh(\theta) e_1+\cosh(\theta) \xi.
\end{equation}
As $U$ is a constant vector field, $\overset{\sim}{\nabla}_{e_2}U=0$ and \eqref{uu} gives
\begin{equation}\label{e21}
\sinh(\theta)\overset{\sim}{\nabla}_{e_2}e_1+\cosh(\theta)\overset{\sim}{\nabla}_{e_2}\xi=0.
\end{equation}
Taking the normal component and using \eqref{vij1}, we obtain
$$\sinh(\theta)\langle \overset{\sim}{\nabla}_{e_2}e_1,\xi\rangle=-\sinh(\theta)\sigma_{21}=0.$$
Since $\theta\not=0$, we conclude $\sigma_{21}=\sigma_{12}=0$. By combining \eqref{vij2} and \eqref{e21}, it follows  that
$$\overset{\sim}{\nabla}_{e_2}e_1=-\coth(\theta) \sigma_{22}\ e_2.$$
Analogously, we have $\overset{\sim}{\nabla}_{e_1}U=0$ and \eqref{uu} yields
$$\sinh(\theta)\ \overset{\sim}{\nabla}_{e_1}e_1+\cosh(\theta)\ \overset{\sim}{\nabla}_{e_1}\xi=0.$$
The normal component of the above expression together \eqref{vij1} gives $\sigma_{11}\sinh(\theta)=0$, that is,
$\sigma_{11}=0$. We can summarize the above computations with a description of $\nabla$ as follows:

\begin{theorem} \label{t1} With the above notations, the  Levi-Civita connection $\nabla$ for  a constant angle spacelike surface in {\sc $\e_1^3$} is given by
\begin{eqnarray*}
& &\nabla_{e_1} e_1=0.\\
& &\nabla_{e_1}e_2=0,\ \nabla_{e_2}e_1=-\coth(\theta)\sigma_{22}\ e_2.\\
& &\nabla_{e_2}e_2=\coth(\theta)\sigma_{22}\ e_1.
\end{eqnarray*}
Moreover, with respect to $\{e_1,e_2\}$, the Weingarten map takes the form
$$\left(\begin{array}{cc}
0&0\\0&-\sigma_{22}
\end{array}
\right).$$
\end{theorem}

\begin{corollary}\label{c1}
 Given a constant angle spacelike surface $M$ in {\sc $\e_1^3$}, there exist local coordinates $u$ and $v$ such that
the metric on $M$ writes as  $\langle~,~\rangle=du^2+\beta^2 dv^2$, where $\beta=\beta(u,v)$ is a smooth function on $M$, i.e. the
coefficients of the first fundamental are $E=1$, $F=0$ and $G=\beta^2$.
\end{corollary}

Now, we will consider that the parametrization $x(u,v)$ given by the above Corollary.
We know that $A(x_u)=0$ and $\sigma_{11}=\sigma_{12}=0$. From Theorem \ref{t1} one obtains
\begin{eqnarray*}
x_{uu}&=&0\\
x_{uv}&=&\frac{\beta_u}{\beta}x_v\\
x_{vv}&=&-\beta\beta_u\ x_u+\frac{\beta_v}{\beta}\ x_v+\beta^2\sigma_{22}\xi
\end{eqnarray*}
On the other hand, we have
\begin{eqnarray*}
\xi_u&=&\overset{\sim}{\nabla}_{x_u}\xi=0.\\
\xi_v&=&\overset{\sim}{\nabla}_{x_v}\xi=\beta\sigma_{22} e_2=\sigma_{22}\ x_v.\\
\end{eqnarray*}
As $\xi_{uv}=\xi_{vu}=0$, it follows $\overset{\sim}{\nabla}_{x_u}(\sigma_{22} x_v)=0$. Using the fact that $\sigma_{12}=0$,
$\overset{\sim}{\nabla}_{x_u}x_v=\overset{\sim}{\nabla}_{x_v}x_u$ and Theorem  \ref{t1},
we obtain
$$0=(\sigma_{22})_u x_v+\sigma_{22}\overset{\sim}{\nabla}_{x_u}x_v=
(\sigma_{22})_u x_v-\coth(\theta)\sigma_{22}^2 x_v.$$
Therefore
\begin{equation}\label{b1}
(\sigma_{22})_u -\coth(\theta)\sigma_{22}^2=0.
\end{equation}
Also, we use the expression of $x_{uv}$ to conclude that
$$(\sigma_{22})_u+\sigma_{22}\frac{\beta_u}{\beta}=0
$$
that is, $(\beta\sigma_{22})_u=0$ and then, there exists a smooth function $\varphi=\varphi(v)$ depending only on $v$ such that
\begin{equation}\label{b2}
\beta\sigma_{22}=\varphi(v).
\end{equation}
 Moreover, by combining  \eqref{b1} and \eqref{b2}, we have
$$\frac{\beta_u}{\beta}=-\coth(\theta)\sigma_{22}.$$

\begin{proposition}\label{pr-f} Consider a constant angle spacelike surface $x=x(u,v)$ in {\sc $\e_1^3$} where $(u,v)$ are
the coordinates given in Corollary $\ref{c1}$. If $\sigma_{22}=0$ on $M$, then $x$ describes an affine plane.
\end{proposition}
\begin{proof}
We know that $\beta_u=0$ on $M$. Thus $x_{uv}=0$ and hence, $x_u$ is a constant vector.
From \eqref{uu}, $\xi$ is a constant vector field along $M$, and so, $x$ parametrizes a (spacelike) plane.
\end{proof}
Here and in the rest of the work, we will assume that  $\sigma_{22}\not=0$.
By solving  equation \eqref{b1}, we obtain  a function $\alpha=\alpha(v)$ such that
$$\sigma_{22}(u,v)=\frac{1}{-\coth(\theta)\  u+\alpha(v)}.$$
Then \eqref{b2} yields
$$\beta(u,v)=\varphi(v)\Big(-\coth(\theta)\  u+\alpha(v)\Big).$$

Consequently,
\begin{eqnarray}
x_{uu}&=&0\label{x1}\\
x_{uv}&=&\frac{\coth(\theta)}{\coth(\theta)u-\alpha(v)}\ x_v\label{x2}\\
x_{vv}&=&\varphi^2\coth(\theta)(-\coth(\theta)u+\alpha)x_u\nonumber\\
&+&\Big(\frac{\varphi'}{\varphi}+\frac{\alpha'}{-\coth(\theta)u+\alpha}\Big)x_v+\varphi^2(-\coth(\theta)u+\alpha)\xi.\label{x3}
\end{eqnarray}
From \eqref{uu} we have
$$\langle x_u,U\rangle=\sinh(\theta),\hspace*{.5cm}\langle x_v,U\rangle=0,$$
or equivalently
$$\langle x,U\rangle_u=\sinh(\theta),\hspace*{.5cm}\langle x,U\rangle_v=0.$$
Then
$$\langle x,U\rangle=\sinh(\theta) u+\mu,\ \mu\in\r.$$
The parametrization of $x$ is now (up to vertical translations)
$$x(u,v)=(x_1(u,v),x_2(u,v),-\sinh(\theta) u).$$
As $E=1$,  there exists a function $\phi:M\rightarrow\r$ such that
$$   x_u=(\cosh(\theta) \cos\phi(u,v),\cosh(\theta)\sin\phi(u,v),-\sinh(\theta) ).$$
Since $x_{uu}=0$, one obtains $\phi_u=0$, that is, $\phi=\phi(v)$ and hence
\begin{eqnarray*}
x_u&=&(\cosh(\theta) \cos(\phi(v)),\cosh(\theta)\sin(\phi(v)),-\sinh(\theta) )\\
&=&\cosh(\theta)(\cos(\phi(v)),\sin(\phi(v)),0)-\sinh(\theta)(0,0,1).
\end{eqnarray*}
Denoting by
$$f(v)=(\cos(\phi(v)),\sin(\phi(v)))$$
we can rewrite $x_u$ as
$$   x_u=\cosh(\theta)(f(v),0)-\sinh(\theta)(0,1).$$
We compute $x_{uv}$:
\begin{equation}\label{xuv}
x_{uv}=\cosh(\theta)(f'(v),0).
\end{equation}
An integration with respect to $u$ leads to
\begin{equation}\label{xv}
x_v=\cosh(\theta)(u f'(v)+h(v),0),
\end{equation}
where $h=h(v)$ is a smooth curve in $\r^2$. From \eqref{x2} and \eqref{xv}
$$x_{uv}=\frac{1}{\coth(\theta)u-\alpha(v)}\frac{\cosh^2(\theta)}{\sinh(\theta)}(u f'(v)+h(v),0).$$
Comparing with \eqref{xuv} one gets
$$h=-\tanh(\theta)\alpha(v) f'(v),$$
and so,
$$x_v=\cosh(\theta)\big(u-\tanh(\theta)\alpha(v)\big)(f'(v),0).$$
 The value of $x_{vv}$ is now
\begin{equation}\label{x33}
x_{vv}=\cosh(\theta)(u-\tanh(\theta)\alpha)(f''(v),0)-\sinh(\theta)\alpha'(f'(v),0).
\end{equation}
Multiplying the two expressions of $x_{vv}$ in \eqref{x3}  and \eqref{x33} by $x_u$, we conclude
$$\phi'(v)=\frac{1}{\sinh(\theta)}\varphi(v).$$
We do a change in the variable $v$ to get $\phi'=1$ for any $v$, that is, $\phi(v)=v$.
It is not difficult to see that this does not change the second derivatives of $x$ in \eqref{x1}, \eqref{x2} and \eqref{x3}. Then
\begin{eqnarray*}
x_u&=&\cosh(\theta)(\cos(v),\sin(v),0)-\sinh(\theta)(0,0,1).\\
x_v&=&\Big(\cosh(\theta)u-\sinh(\theta)\alpha(v)\Big)(-\sin(v),\cos(v),0).
\end{eqnarray*}

The above reasoning can be written by the following

\begin{theorem}\label{t-main}
\label{th:char}
Let $M$ be a constant angle spacelike surface in  Minkowski space {\sc $\e^3_1$} which is not totally geodesic.
Up to a rigid motion of the ambient space,  there exist local coordinates $u$ and $v$ such that $M$ is given
by the parametrization
\begin{equation}
\label{eq:param}
x(u,v)=\Big(u\cosh(\theta)\big(\cos(v),\sin(v)\big)+\psi(v),-u\sinh(\theta)\Big)
\end{equation}
with
\begin{equation}\label{eq-param2}
\psi(v)=\sinh(\theta)\Big(\int\alpha(v)\sin(v),-\int\alpha(v)\cos(v)\Big)
\end{equation}
where $\alpha$ is a smooth function on a certain interval $I$.
{\rm Here $\theta$ is the hyperbolic angle between the unit normal at $M$ and the fixed direction $U=(0,0,1)$.}
\end{theorem}

\begin{proposition}
A constant angle spacelike surface is  flat.
\end{proposition}

\begin{proof} At each point $p\in M$, we consider $\{v_1(p),v_2(p)\}$  a basis of eigenvectors of the Weingarten endomorphism $A_p$. In particular,
$\lambda_i(p)=-\sigma_{ii}(p)$.
As the function $\langle \xi,U\rangle$ is constant, a differentiation along $v_i(p)$ yields
$\langle \overset{\sim}{\nabla}_{v_i(p)}\xi,U\rangle=0$, $i=1,2$. Using \eqref{vij2}, we obtain
$$\lambda_1(p)\langle v_1(p),U\rangle=\lambda_2(p)\langle v_2(p),U\rangle=0.$$
Assume that at the point $p$, $\lambda_1(p)\lambda_2(p)\not=0$.
By using the  continuity of the principal curvature functions, we have
 $\langle v_1(q),U\rangle=\langle v_2(q),U\rangle=0$ for every point $q$ in a neighbourhood ${\cal N}_p$ of $p$.
 This means that $U$ is a normal vector in ${\cal N}_p$ and hence it follows $\theta=0$: contradiction. Thus
$\lambda_1(p)\lambda_2(p)=0$ for any $p$, that is, $K=0$ on $M$.
\end{proof}

As in  Euclidean space,  all flat surfaces are  characterized to be
locally isometric to planes, cones, cylinders or tangent developable
surfaces.

\begin{corollary}
\label{flat}
Any constant angle spacelike surface is isometric to a plane, a cone, a cylinder or a tangent developable surface.
\end{corollary}

The fact that a constant angle (spacelike) surface is a ruled surface appears in Theorem \ref{t-main}. Exactly, the parametrization
\eqref{eq:param} writes as
$$x(u,v)=(\psi(v),0)+u\Big(\cosh(\theta)\big(\cos(v),\sin(v)\big),-\sinh(\theta)\Big),$$
which proves that our surfaces are ruled. Next we present some examples of surfaces obtained in Theorem \ref{t-main}.

{\bf Example 1.} We take different choices of the function $\alpha$ in \eqref{eq-param2}.
\begin{enumerate}
\item Let $\alpha(v)=0$. After a change of variables, $\psi(v)=(0,0)$ and
$$x(u,v)=u(\cosh(\theta)(\cos(v),\sin(v)),-\sinh(\theta)).$$
This surface is a cone with the vertex the origin and whose basis curve is a circle in a horizontal plane. See Figure \ref{fig1}, left.
\item Let $\alpha(v)=1$. Then $\psi(v)=-\sinh(\theta)(\cos(v),\sin(v))$ and
$$x(u,v)=-\sinh(\theta)(\cos(v),\sin(v),0)+u(\cosh(\theta)(\cos(v),\sin(v)),-\sinh(\theta)).$$
Again, this surface is a cone  based in a horizontal circle.
\item Consider $\alpha(v)=1/\sin(v)$. Then $\psi(v)=\sinh(\theta)(v,-\log(|\sin(v)|))$ and
$$x(u,v)=\sinh(\theta)(v,-\log(|\sin(v)|),0)+u(\cosh(\theta)(\cos(v),\sin(v)),-\sinh(\theta)).$$
 See Figure \ref{fig1}, right.
\end{enumerate}
\begin{figure}[htbp]
\begin{center}
\includegraphics[width=5cm]{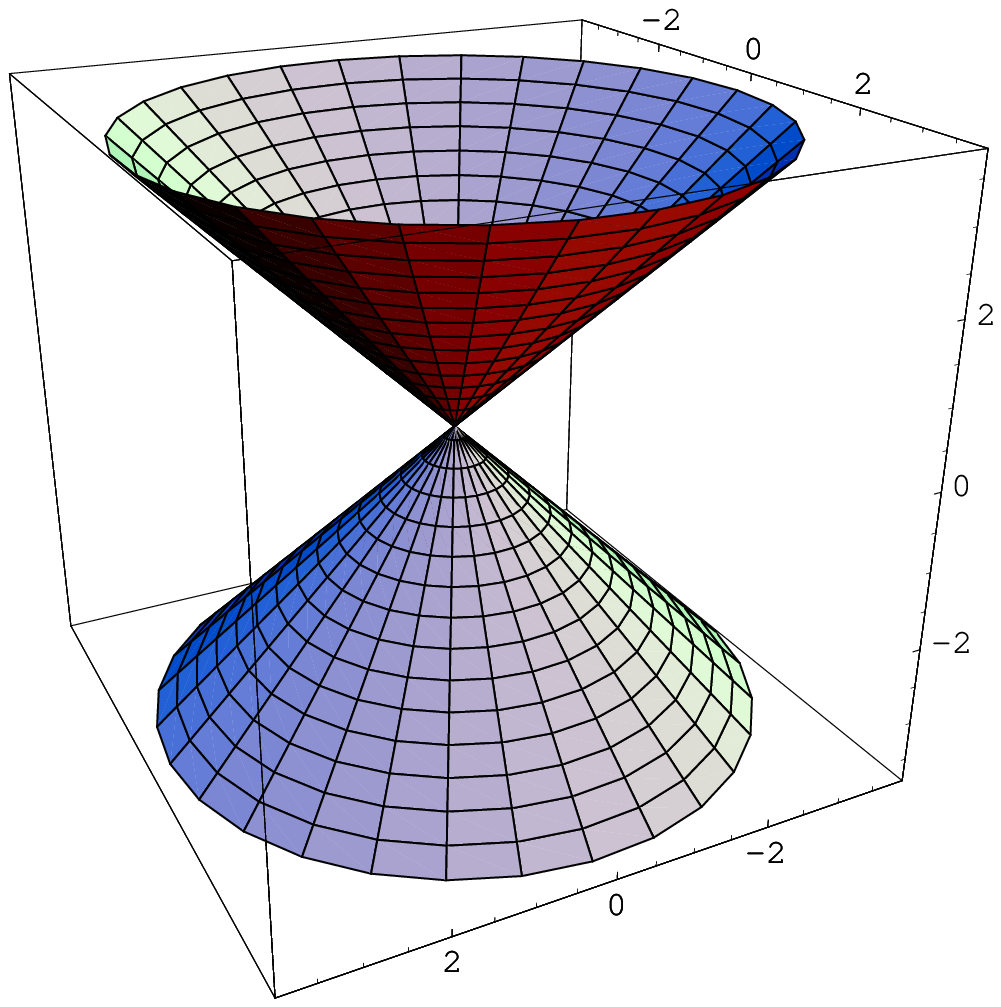}
\hspace*{2cm}
\includegraphics[width=6cm]{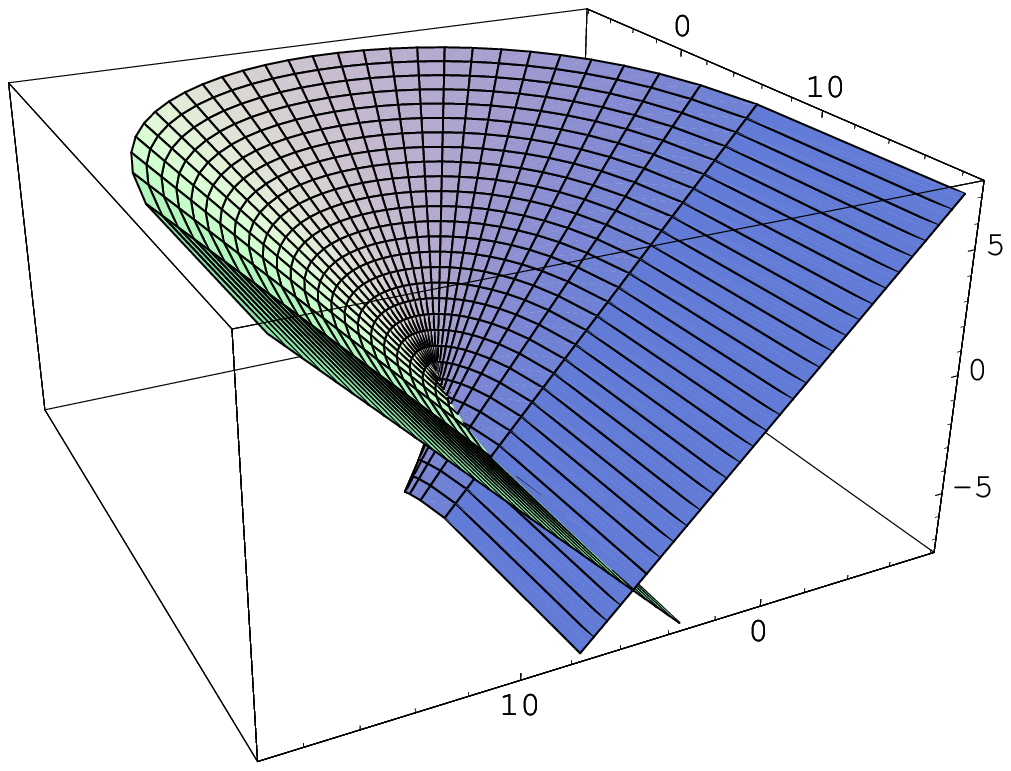}
\end{center}
\caption{Constant angle surfaces corresponding to several  choices of $\alpha$ in Theorem \ref{t-main}: $\alpha(v)=0$ (left) and $\alpha(v)=1/\sin(v)$ (right).}
\label{fig1}
\end{figure}
\section{Tangent developable constant angle surfaces}\label{s-tangent}

 In this section we study tangent developable surfaces that are constant angle surfaces
(see \cite{kn:Nis09} for the Euclidean ambient space). Given a regular curve $\gamma:I\rightarrow\e_1^3$, we define the tangent surface $M$ generated  by $\gamma$ as the surface parameterized by
$$x(s,t)=\gamma(s)+t\gamma'(s),\ \ (s,t)\in I\times\r.$$
The tangent plane at a point $(s,t)$ of $M$ is spanned by $\{x_s,x_t\}$, where
$$x_s=\gamma'(s)+t\gamma''(s),\hspace*{1cm}x_t=\gamma'(s).$$
The surface is regular at those points where $t(\gamma'(s)\times\gamma''(s))\not=0$. Without loss of generality, we will assume that $t>0$.

On the other hand, since $M$ is a  spacelike surface and $\gamma(s)\in M$, the curve $\gamma$ must be   spacelike. We parametrize
 $\gamma$ such that $s$ is the arc-length parameter, that is, $\langle\gamma'(s),\gamma'(s)\rangle=1$ for every $s$. As a consequence,
 $\gamma''(s)$ is orthogonal to $\gamma'(s)$. We point out that although $\gamma$ is a spacelike curve, the acceleration vector
 $\gamma''(s)$ can be of any causal character, that is, spacelike, timelike or lightlike. However, the surface $M$ is spacelike, which
 implies that $\gamma$ is not an arbitrary curve. Indeed, by computing the  first fundamental form $\{E,G,F\}$ of $M$ with respect to basis $\{x_s,x_t\}$, we obtain
$$\left(\begin{array}{cc}
E&F\\F&G
\end{array}\right)(s,t)=\left(\begin{array}{cc}
1+ t^2\langle\gamma''(s),\gamma''(s)\rangle&1\\1&1
\end{array}\right).$$
Thus $M$ is spacelike if and only if $EG-F^2>0$. This is equivalent to $\langle\gamma''(s),\gamma''(s)\rangle>0$, that is,  $\gamma''(s)$ is spacelike for any $s$.

The tangent vector $\t(s)$ and the normal vector $\n(s)$ are defined by
$\t(s)=\gamma'(s)$, $\n(s)=\gamma''(s)/\kappa(s)$, respectively, where $\kappa(s)=|\gamma''(s)|>0$ is the curvature of $\gamma$ at $s$.
The Frenet Serret frame of $\gamma$ at each point $s$ associates an orthonormal basis  $\{\t(s),\n(s),\b(s)\}$, where
 $\b(s)=\t(s)\times\n(s)$ is called the binormal vector (\cite{kn:Kuh,kn:lo}). We remark that $\b(s)$ is a unit timelike vector. The corresponding Frenet equations are
$$
\left\{\begin{array}{rccc}
\t'= & & \kappa \n & \\
\n'= & -\kappa\t & & +\tau\b\\
\b'= & & \tau \n. &
\end{array}\right.
$$
The function  $\tau(s)=-\langle \n'(s),\b(s)\rangle$ is called the torsion of $\gamma$ at $s$. For tangent surfaces $x$, the unit normal vector field $\xi$ to
$M$ is  $\xi=(x_s\times x_t)/\sqrt{EG-F^2}=- \b$.

In order to give the next result, recall the concept of a helix in Minkowski space. A spacelike (or timelike) curve $\gamma=\gamma(s)$
parameterized by the arc-length is called \emph{a helix} if  there exists a vector  $v\in\e_1^3$ such that the function
$\langle\gamma'(s),v\rangle$ is constant. This is equivalent to say that the function $\tau/\kappa$ is constant.

\begin{theorem} \label{tangent}
Let $M$ be a tangent developable spacelike surface generated by $\gamma$. Then $M$ is a constant angle surface if and only if $\gamma$
is a helix with $\tau^2<\kappa^2$. Moreover the direction $U$ with which $M$ makes a constant hyperbolic angle $\theta$ is given by
\begin{equation}\label{td1}
U=\frac{1}{\sqrt{\kappa^2-\tau^2}}\Big(-\tau(s) \t(s)+\kappa(s)\b(s)\Big)
\end{equation}
and the angle $\theta$ is determined by the relation
\begin{equation}\label{td2}
\cosh(\theta)=\frac{\kappa}{\sqrt{\kappa^2-\tau^2}}.
\end{equation}
\end{theorem}

\begin{proof}
\begin{enumerate}
\item Assume that  $M$ makes a constant angle with a fixed direction $U$, with $\langle U,U\rangle=-1$. Then $\langle \b(s),U\rangle$ is a
constant function $c$ with $c<0$. By differentiation with respect to $s$, and using the Frenet equation, we have
$\tau\langle\n(s),U\rangle=0$ for any $s$. If $\langle\n(s_0),U\rangle\rangle\not=0$ at some point $s_0$, then $\tau=0$ in a neighborhood of $s_0$.
This means that the binormal $\b(s)$ is a constant vector $U$, $\gamma$ is a planar curve and $\xi=-U$ is constant on $M$. In particular,
$\gamma$ is a helix with $\tau^2<\kappa^2$ and the surface is a (spacelike) affine plane. Equations \eqref{td1} and \eqref{td2} are trivial.

If $\langle\n(s),U\rangle=0$ on $I$, and because $\langle U,U\rangle=-1=\langle\t(s),U\rangle^2-c^2$,  the function
$\langle\t(s),U\rangle$ is a constant function. Therefore $\gamma$ is a helix in  $\e_1^3$ again. A differentiation of $\langle\n(s),U\rangle=0$ gives
$\langle\t(s),U\rangle=c\tau/\kappa$. Thus $-1=c^2\tau^2/\kappa^2-c^2$, which shows that $\tau^2<\kappa^2$. Moreover,
$c=-\kappa/\sqrt{\kappa^2-\tau^2}$. As $U=\langle\t(s),U\rangle\t(s)-c\b(s)$,  we get the expression \eqref{td1}. Finally (\ref{td2}) is trivial.

\item Conversely, let $\gamma=\gamma(s)$ be a helix and let $x=x(s,t)$ be the corresponding tangent surface.
We know that $\tau/\kappa$ is a constant function. If $\tau=0$,  $\gamma$ is a planar curve.
Then the tangent surface generated by $\gamma$ is a plane, which is a constant angle surface. If $\tau\not=0$, let us define
$$U(s)=-\frac{\tau}{\kappa}\t(s)+\b(s).$$
Using the Frenet equations, we have $dU/ds=0$, that is, $U$ is a constant vector. Moreover,
$\langle \xi,U\rangle=-\langle\b(s),U\rangle=1$. Thus $M$ is a constant angle surface. The hyperbolic angle $\theta$ is given by
$$\cosh(\theta)=\frac{\langle\xi,U\rangle}{\sqrt{-\langle U,U\rangle}}=\frac{\kappa}{\sqrt{\kappa^2-\tau^2}}.$$
\end{enumerate}
\end{proof}

We present two examples of constant angle surfaces that are tangent surfaces. After an isometry of the ambient space, we assume that
$U=E_3$. From \eqref{td2} if $\tau/\kappa=a$, with $|a|<1$, then
$\cosh(\theta)=1/\sqrt{1-a^2}$. Moreover $\langle\t(s),U\rangle=-\sinh(\theta)$ and
$\langle\gamma(s),E_3\rangle=-\sinh(\theta)s+b$, with $b\in\r$. After an appropriate change of variables, we take $b=0$ and we write
$$\gamma(s)=(\gamma_1(s),\gamma_2(s),\sinh(\theta)s).$$
 Because $s$ is the arc-length parameter, there exists a smooth function $\lambda(s)$ such that
 $\gamma'(s)=(\cosh(\theta)\cos(\lambda(s)),\cosh(\theta)\sin(\lambda(s)),\sinh(\theta))$. An easy computation leads to
\begin{eqnarray*}
\n(s)&=&(-\sin(\lambda(s)),\cos(\lambda(s)),0)\\
\b(s)&=&(-\sinh(\theta)\cos(\lambda(s)),-\sinh(\theta)\sin(\lambda(s)),-\cosh(\theta)).
\end{eqnarray*}
The curvature is $\kappa(s)=\cosh(\theta)\lambda'(s)$ and the torsion is $\tau(s)=-\sinh(\theta)\lambda'(s)$.


{\bf Example 2.} We take $\lambda(s)=s$. An integration yields
$$\gamma(s)=(\cosh(\theta)\sin(s),-\cosh(\theta)\cos(s),\sinh(\theta)s).$$
 Here $\kappa(s)=\cosh(\theta)$ and
$\tau(s)=-\sinh(\theta)$ and  $\gamma$ is a helix where both the curvature and torsion functions are constant.
A picture of the curve $\gamma$ and the corresponding tangent surface appears in Figure \ref{fig2}.

\begin{figure}[htbp]
\begin{center}
\includegraphics[width=3cm]{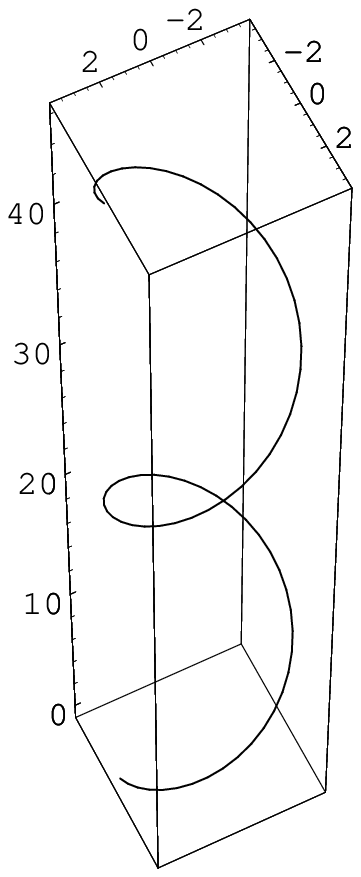}
\hspace*{2cm}
\includegraphics[width=6cm]{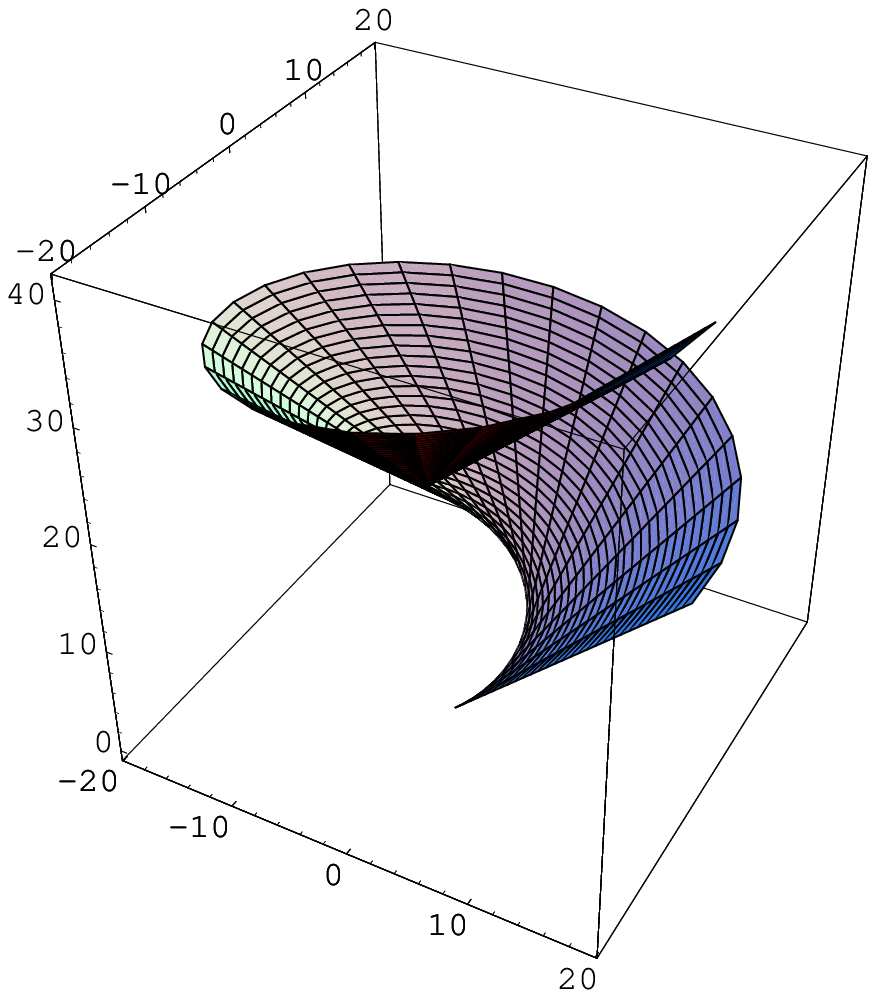}
\end{center}
\caption{A  constant angle tangent developable surface with $\kappa(s)=\cosh(\theta)$ and $\tau(s)=-\sinh(\theta)$.
Here $\theta=2$ and $U=(0,0,1)$. }
\label{fig2}
\end{figure}

{\bf Example 3.} We take $\lambda(s)=s^2$. Recall that the Fresnel functions are defined as
$$\mbox{FrS}(x)=\int_0^x \sin\big(\frac{\pi t^2}{2}\big)dt\hspace*{1cm}\mbox{FrC}(x)=\int_0^x \cos\big(\frac{\pi t^2}{2}\big)dt.$$
Then
$$\begin{array}{c}
\gamma(s)=\left(\sqrt{\frac{\pi}{2}}\cosh(\theta) \mbox{FrC}\big(\sqrt{\frac{2}{\pi}}~s\big),\sqrt{\frac{\pi}{2}}\cosh(\theta)
\mbox{FrS}\big(\sqrt{\frac{2}{\pi}}~s\big),\sinh(\theta)s\right)
\end{array}
$$
is a helix where  $\kappa(s)=2\cosh(\theta)s$ and $\tau(s)=-2\sinh(\theta)s$.
Figure \ref{fig3} shows the curve $\gamma$ and the generated tangent surface.

\begin{figure}[htbp]
\begin{center}
\includegraphics[width=3cm]{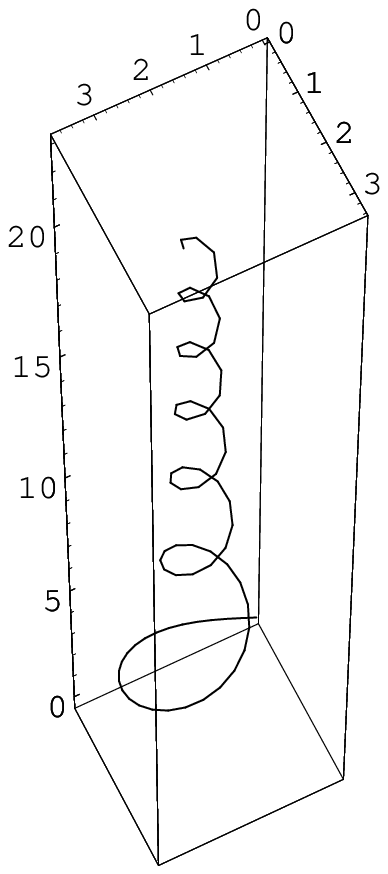}
\hspace*{2cm}
\includegraphics[width=6cm]{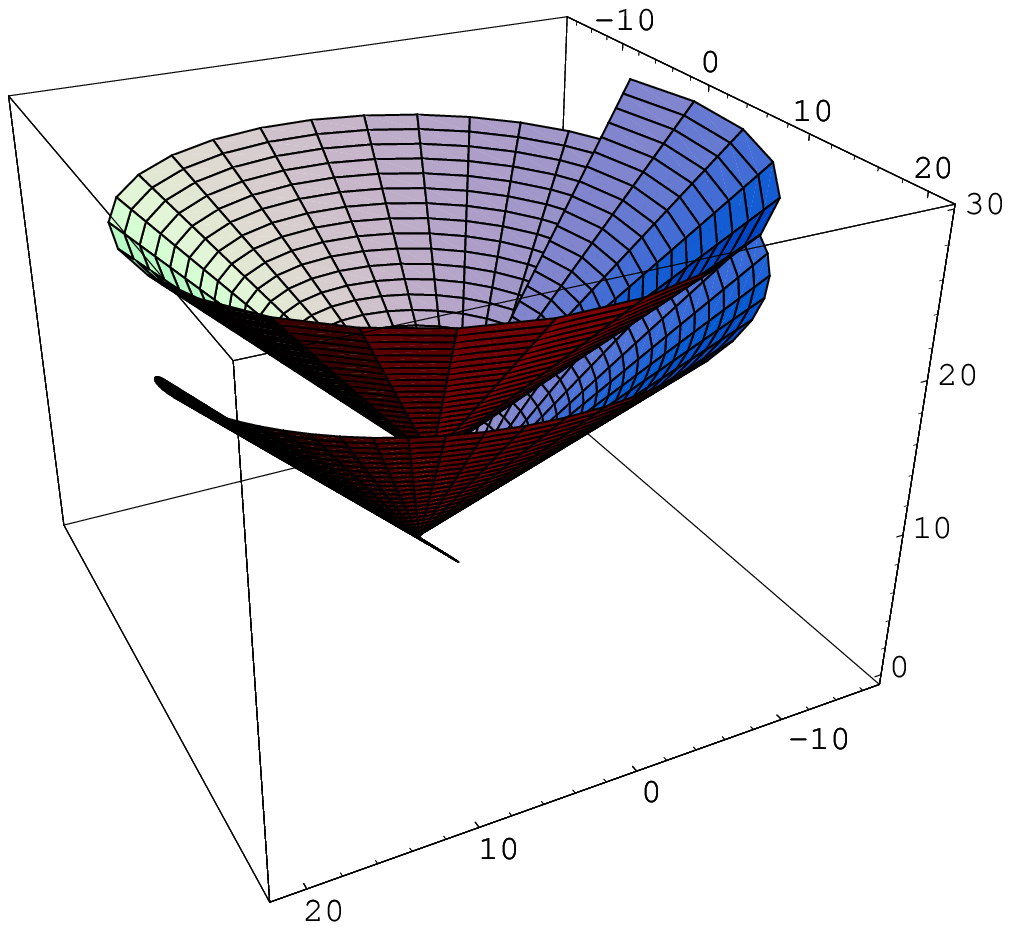}
\end{center}
\caption{A  constant angle tangent developable surface with $\kappa(s)=2s\cosh(\theta)$ and $\tau(s)=-2s\sinh(\theta)$. Here
$\theta=2$ and $U=(0,0,1)$. }\label{fig3}
\end{figure}

{\bf Remark.} We can extend the concept of constant angle surfaces for tangent developable \emph{timelike} surfaces.
Let $M$ be a tangent surface generated by a curve $\gamma$ such that $M$ is timelike. Then $\gamma$ is a spacelike
curve (with $\gamma''$ timelike) or $\gamma$ is a timelike curve (with  $\gamma''$ spacelike). Assume that $\gamma$
is parameterized by the arc-length $s$. Denote by $\{\t,\n,\b\}$ the Frenet frame of $\gamma$, that is,
$\t(s)=\gamma'(s)$, $\n(s)=\gamma''(s)/\kappa(s)$, with $\kappa(s)=|\gamma''(s)|$ and $\b(s)=\t(s)\times\n(s)$.
The Frenet equations are
$$
\left\{\begin{array}{cccc}
\t'= &  &\kappa \n &\\
\n'= & \kappa\t & & +\tau\b\\
\b'= &  &\epsilon\tau \n&
\end{array}\right.
$$
where $\tau=\langle \n',\b\rangle$ and $\langle\t(s),\t(s)\rangle=\epsilon=-\langle\n(s),\n(s)\rangle$, $\epsilon\in\{1,-1\}$. Anyway,
$\b$ is always spacelike. We assume  that there exists a fixed vector
$U\in\e_1^3$ such that the function $\langle\xi,U\rangle$ is constant. Then it is not difficult to show that this condition is
equivalent to say that $\gamma$ is a planar curve ($\tau=0$, and  $M$ is an affine plane), or $\langle\n(s),U\rangle=0$ for any $s$.
In this case, the first Frenet equation yields $\langle \t',U\rangle=0$
and thus, $\langle\t(s),U\rangle$ is a constant function.
This means that $\gamma$ is a helix of $\e_1^3$. This generalizes Theorem \ref{tangent} for tangent timelike surfaces.

We point out that our parametrization of $M$,
$x(s,t)=\gamma(s)+t\gamma'(s)$ where $\gamma$ is a helix given by
$$
\gamma(s)=\left(\cosh(\theta)\int\cos(\lambda(s)), \cosh(\theta)\int\sin(\lambda(s)), \sinh(\theta)s\right)
$$
does not satisfy the conditions of Corollary \ref{c1} since $F\not=0$.
In order to obtain the parametrization given in Theorem \ref{t-main}, we do
a change of parameters given by
$$u=-(s+t) ~,~ v=\pi+\lambda(s).$$
Now we obtain $x_s=-x_u+\lambda' x_v$ and $x_t=-x_u$.

But $x_t=(\cosh(\theta)\cos(\lambda(s)),\cosh(\theta)\sin(\lambda(s)),\sinh(\theta))$ or, in terms of $u$ and $v$
$$
x_u=(\cosh(\theta)\cos(v),\cosh(\theta)\sin(v),-\sinh(\theta)).
$$
Similarly $x_s=x_t+t\lambda'(s)~(-\cosh(\theta)\sin(\lambda(s)),\cosh(\theta)\cos(\lambda(s)),0)$.
It follows
$$
x_v=\big(u+\lambda^{-1}(v-\pi)\big)\cosh(\theta)~(-\sin(v),\cos(v),0).
$$
Consequently, the function $\alpha$ involved in the general formula can be expressed as
$$
\alpha(v)=-\coth(\theta)~\lambda^{-1}(v-\pi).
$$

\section{Constant angle cylinders and cones}\label{s-cylinder}

In this section we consider cylinders and cones that are constant angle (spacelike) surfaces.
A ruled surface is called a \emph{cylinder} if it can be parameterized by $x(s,t)=\gamma(s)+tv$,
where $\gamma$ is a regular curve and $v$ is a fixed vector. The regularity of the cylinder is given by the
fact that $\gamma'(s)\times v\not=0$. A \emph{cone} is a ruled surface that can be parameterized by
$x(s,t)=t\gamma(s)$, where $\gamma$ is a regular curve. The vertex of the cone is the origin and the surface is regular wherever $t\big(\gamma(s)\times\gamma'(s)\big)\not=0$.

\begin{theorem}\label{t-cylinder} The only  constant angle (spacelike) cylinders are  planes.
\end{theorem}

\begin{proof} Let $M$ be a spacelike cylinder generated by a curve
$\gamma$ and a fixed direction $v$. As the surface is spacelike, $v$ is a
spacelike vector, which it will be assumed $|v|=1$. We can suppose that
$\gamma$ is contained in a plane $\Pi$ being  $v$ orthogonal to $\Pi$. In
particular, $\Pi$ is a timelike plane. The unit normal vector is  $\xi(s,t)=\xi(s)=\gamma'(s)\times v$.

By contradiction, we assume that $\gamma$ is not a straight-line, that is,
$\kappa(s)\not=0$ at some interval. We consider $\{\t,\n,\b\}$ the Frenet frame of $\gamma$. As $\gamma$ is a planar curve,
$\b(s)=\pm v$ and so, $\xi(s)=\pm\n(s):=\gamma''(s)/\kappa(s)$. Let $U$ be the unit (timelike) vector such that
the function $\langle\xi(s),U\rangle$ is constant, that is, $\langle\n(s),U\rangle$ is constant. By differentiation with respect to $s$,
using the Frenet equations and since $\gamma$ is a planar curve, we obtain
$\langle\t(s),U\rangle=0$ for any $s$. A new differentiation  gives  $\kappa(s)\langle\n(s),U\rangle=0$ for any $s$. As $\kappa(s)\not=0$, we have
$\langle\n(s),U\rangle=0$, for any $s$. However,   $\n(s)$ and $U$
are both timelike vectors and thus, the product $\langle\n(s),U\rangle$ can never vanish: contradiction. Consequently, $\kappa(s)=0$ for any $s$, that is,
$\gamma$ is a straight-line and then $M$ is a (spacelike) plane.
\end{proof}

\begin{remark} We point out that this result is more restrictive than the corresponding in Euclidean space {\sc $\e^3$}. In {\sc $\e^3$}, any cylinder is a
constant angle: it is suffices to take $U$ as the vector that defines the
rulings of the cylinder. The difference in  Lorentzian ambient is that our surfaces are spacelike and the vector $U$ is timelike, which imposes extra conditions.
\end{remark}

For the next result for cones,  we need recall that a (spacelike) circle in Minkowski space is a planar curve with constant curvature \cite{kn:lls,kn:lo}.
We also point that the plane $\Pi$ containing the circle can be of any causal character. Indeed, after a rigid motion of $\e_1^3$,
a spacelike circle can be viewed as follows: a Euclidean circle in a horizontal plane (if $\Pi$ is spacelike), a hyperbola in a vertical plane
(if $\Pi$ is timelike) and a parabola in a $\pi/4$-inclined plane (if $\Pi$ is lightlike).

\begin{theorem} \label{t-cone} Let $M$ be a (spacelike) cone. Then $M$ is a constant angle surface if and only if the generating curve is a circle in a
spacelike plane or it is a straight-line (and $M$ is a plane).
\end{theorem}
\begin{proof}
Let $M$ be a cone, which can assume that its vertex is the origin of $\r^3$. Let $x(s,t)=t\gamma(s)$ be a parametrization of $M$, where $t\not=0$ and
$\gamma(s)\not=0$, $s\in I$. As $x_s=t\gamma'(s)$ is spacelike, $\langle\gamma'(s),\gamma'(s)\rangle>0$. On the other hand, $x_t$ must be spacelike,
this means that $\langle\gamma(s),\gamma(s)\rangle>0$. We can change $\gamma(s)$ by a proportional vector and to assume that $\gamma$ lies in the unit Minkowski sphere of
$\e_1^3$, that is, in the De Sitter space $\s^2_1=\{x\in\e_1^3;x_1^2+x_2^2-x_3^2=1\}$. Thus, $|\gamma(s)|=1$ for any $s\in I$. Without loss of generality,
we suppose that $\gamma=\gamma(s)$ is parameterized by the arc-length. Then $\gamma(s)$ and $\gamma''(s)$ are orthogonal to
$\gamma'(s)$. The unit normal vector field $\xi$ on $M$ is collinear to $x_s\times x_t$. Denoting by $\t(s)=\gamma'(s)$, we have
$\xi=\t(s)\times\gamma(s)$. In particular,
\begin{equation}\label{eq-cone0}
\gamma''(s)=-\gamma(s)-\langle\gamma''(s),\xi(s)\rangle\xi(s).
\end{equation}
  Assume that $M$ is a constant angle surface and let
$U$ be the unit timelike vector such that
$\langle\xi(s),U\rangle$ is constant. By differentiation with respect to $s$, we
have
\begin{equation}\label{eq-cone}
\langle\gamma''(s)\times\gamma(s),U\rangle=0
 \end{equation}
for any $s$. Substituting in \eqref{eq-cone} the value of $\gamma''(s)$ obtained in \eqref{eq-cone0}, we get
$$\langle\gamma''(s),\gamma'(s)\times\gamma(s)\rangle\langle\gamma'(s),U\rangle=0.$$
 We discuss the two possibilities:
\begin{enumerate}
\item If $\langle\gamma''(s),\gamma'(s)\times\gamma(s)\rangle\not=0$ at some point, then
$\langle\gamma'(s),U\rangle=0$ for any $s$. This means that $\gamma(s)$ lies in a plane orthogonal to $U$ and so, this plane must be spacelike.
Thus the acceleration $\gamma''(s)$ is a spacelike vector. Then we can take the Frenet frame of $\gamma$, namely $\{\t,\n,\b\}$, where
$\b=\t\times\n$ is a timelike vector. Moreover, $\b(s)=\pm U$.
If $\kappa(s)=0$ for any $s$, then $\gamma$ is a straight-line and the surface is a plane. On the contrary,
since $\langle\t(s),\gamma(s)\rangle=0$, by taking the derivative, one obtains
$\kappa(s)\langle\n(s),\gamma(s)\rangle+1=0$. On the other hand, because $\gamma$ is a planar curve ($\tau=0$),
 the derivative of the function $\langle \n(s),\gamma(s)\rangle$ vanishes. This means that $\langle\n(s),\gamma(s)\rangle$ is constant and so, $\kappa(s)$ is constant.
\item Assume $\langle\gamma''(s),\gamma'(s)\times\gamma(s)\rangle=0$ for any $s$. As
$\gamma(s)$ and $\gamma'(s)$ are orthogonal spacelike vectors, then $\gamma''(s)$ is a spacelike vector.
Again, we consider the Frenet frame $\{\t,\n,\b\}$ where $\b$ is a timelike vector. The above equation writes now
as $\kappa(s)\langle\b(s),\gamma(s)\rangle=0$. If $\kappa(s)=0$ for any $s$, then $\gamma$ is a straight-line again.
Suppose now $\langle\b(s),\gamma(s)\rangle=0$. Similar to the previous case,
because $\gamma(s)\in \s^2_1$, it follows $\langle\t(s),\gamma(s)\rangle=0$ and $\kappa(s)\langle\n(s),\gamma(s)\rangle+1=0$. In particular,
$\langle\n(s),\gamma(s)\rangle\not=0$ and then, the derivative of $\langle\b(s),\gamma(s)\rangle$ implies $\tau=0$, that is,
$\gamma$ is a planar curve. Finally, the derivative of $\langle\n(s),\gamma(s)\rangle$ is zero, namely
$\langle\n(s),\gamma(s)\rangle$ is constant, and then, $\kappa(s)$ is constant too.
\end{enumerate}

\end{proof}

As an example of constant angle cones, Figure \ref{fig1} shows a cone based in circle contained in a (horizontal) spacelike plane.

\end{document}